\documentclass[12pt]{article}
\usepackage{amssymb}
\usepackage{mathrsfs}
\usepackage{amsfonts}

\newenvironment{proof}{\noindent {\it Proof.~~}\ }{\  \rule{1mm}{2mm}\medskip}

\newenvironment{proofof}[2]{\noindent {\it Proof of #1}~#2: \
}{~\rule{1mm}{2mm}\medskip}

\newtheorem{theorem}{Theorem}[section]
\newtheorem{lemma}[theorem]{Lemma}

\newtheorem{theirtheorem}{Theorem}

\newtheorem{theirlemma}[theirtheorem]{Lemma}

\usepackage{latexsym,amssymb}
\usepackage{amsmath}

\begin{document}

\title{ Distinct Lengths Modular Zero-sum Subsequences: \\
A Proof of Graham's Conjecture  }

 \author{Weidong Gao, \thanks{Center for Combinatorics, LPMC, Nankai University,
 Tianjin, 300071, People's Republic of China} \ {Y. O. Hamidoune,}\thanks{
UPMC Univ Paris 06,
 E. Combinatoire, Case 189, 4 Place Jussieu,
75005 Paris, France.}
 \ Guoqing Wang \thanks{ Institute of Mathematics,
 Dalian University of Technology, Dalian, 116024, People's Republic of China}}
 \maketitle

\begin{abstract}
Let $n$ be a positive integer and let $S$ be  a sequence of $n$
integers in the interval $[0,n-1]$. If   there is an $r$ such that
any nonempty subsequence with  sum $\equiv 0$ $\pmod n$ has length
$=r,$ then $S$ has at most two distinct values. This proves a
conjecture of R. L. Graham. A previous result of P. Erd\H{o}s and E.
Szemer\'edi shows the validity of this conjecture if $n$ is a large
prime number.
\end{abstract}


\section{Introduction and main result}

We quote:

{`` Graham stated the following conjecture:

Let $p$ be a prime and $a_1, \ldots , a_p$ \ $p$ non-zero residues $\pmod p.$  Assume that if
$\sum\limits_{i=1}^p
\epsilon _ia_i$, $\epsilon _i=0$ or $1$ (not all $\epsilon _i=0$) is a multiple of $p$ then $\sum\limits_{i=1}^p
\epsilon _i$ is uniquely determined. The conjecture states that there only two distinct residues among the $a$'s.
We are going to prove this conjecture for all sufficiently large $p$. In fact we will give a sharper result.
To extend our proof for the small values of $p$ would require considerable computation, but no theoretical difficulty.
 Our proof is surprisingly complicated and we are not convinced that a simpler proof is not possible, but we could not
 find one.  (P. Erd\H{o}s and E. Szemer\'edi \cite{ES})''

The conviction that a simple proof must exist was restated by Erd\H{o}s and Graham in \cite{EG}.

In this work, we prove Graham's Conjecture for non necessarily prime noduli. Since our proof uses an ingredient (proved by elementary methods, but not very shortly), it could not
 the simple proof whose existence is suspected by Erd\H{o}s and  Szemer\'edi. Actually the Erd\H{o}s-Szemer\'edi Theorem
may be formulated equivalently as a modular Zero-sum statement:

\medskip
 \begin{theirtheorem} (Erd\H{o}s-Szemer\'edi \cite{ES}){\sl Let $p$ be a sufficiently large prime and let $S$ be  a sequence
of $p$ integers in the interval $[1,p-1]$. If   there is an $r$ such that
any nonempty subsequence with  sum $\equiv 0$ $\pmod p$ has length $=r,$
then $S$ has at most two distinct values.}
\end{theirtheorem}

In this paper, we obtain the following generalization of this result:

\medskip

\begin{theorem}\label{Theorem:main resul} { \sl Let $n$ be a positive integer and let $S$ be  a sequence
of $n$ integers in the interval $[0,n-1]$. If   there is an $r$ such that
any nonempty subsequence with  sum $\equiv 0$ $\pmod n$ has length $=r,$ then $S$ has at most two distinct values.}
\end{theorem}

In the investigation of Zero-sum sequences in an abelian group $G$, it
is quite convenient to work with an unordered sequence. This is usually done by  identifying a sequence  with an element  of the free abelian monoid generated by $G$.
This point of view together with the bases of Zero-sum Theory are presented in the text book of Geroldinger-Halter-Koch \cite{gerlodinger}.

One  may also define a sequence as a word. In this case, multiplication is just juxtaposition and thus $x^n$
is the word ${x, \ldots ,x}.$ We shall present our proofs in
such a way to fit with each of these definitions.

We give below examples of sequences with a unique non-empty length for modular zero-sum sequences.
\begin{itemize}
  \item $S=1^{n-1}x,$ where $x$ is an integer.
  \item $S=1^{n-2}{(q+1)}^2,$ where $n=2q+1.$
  \item $S=2^{q+r}{1}^{q-r},$ where $n=2q$ and $r$ is odd.
\end{itemize}

\section{Preliminaries}

Let $T$ be a subsequence of a sequence $S.$ We shall denote the sequence obtained from $S$ by deleting $T$ by
$ST^{-1}.$ The sum of elements of $S$ will be denoted by  $\sigma(S).$   The maximal repetition
of a value of $S$  will be denoted by $h(S)$.

We present below few tools:

\begin{theirlemma} (folklore) \label{Lemma:Alon} A sequence $S$ of $n$ integers in the interval $[0,n-1]$   has a  nonempty subsequence with length $\le h(S)$ and sum   $\equiv 0$  $\pmod n.$

\end{theirlemma}

 Lemma \ref{Lemma:Alon} is a special case
of Conjecture 4 of  Erd\H{o}s and  Heilbronn \cite{EH}. In a note added in proofs, Erd\H{o}s and  Heilbronn \cite{EH}
mentioned that Flor proved this conjecture using the Moser-Scherck's Theorem \cite{sch}.

The next Lemma is just an exercise:

\begin{theirlemma} (folklore) \label{n-1} A   sequence of $n-1$ integers in the interval $[0,n-1],$
assuming two distinct values, has a nonempty subsequence with sum   $\equiv 0$ $\pmod n.$
\end{theirlemma}

Let $S=a_1\cdot \ldots \cdot a_t$ be a sequence of integers. We write $m*S=(ma)\cdot \ldots \cdot (ma_t)$. The following  result is a basic tool in our approach:

\begin{theirtheorem}(\cite{SC}, \cite{Yuan})\label{Lemma:SavchevChen} Let $t$ be positive integer with $t\geq \frac{n+1}{2}.$ Let $a_1,\cdot \ldots\cdot ,a_t $ be integers and put $T=a_1\cdot \ldots \cdot a_t.$ If $T$ has no nonempty subsequence with sum   $\equiv 0$ $\pmod n.$
 Then there exists an integer $m$
co-prime to $n$ and positive integers $b_1, \ldots ,b_t$ such that
 $m*S=b_1\cdot \ldots \cdot b_t$
 and  $b_1+\ldots+ b_t<n$.

\end{theirtheorem}

\section{Proof of the main result}

We start with one lemma:

\begin{lemma}\label{Lemma:Our lemma} Let $S=1^{v} a_1\cdot \ldots \cdot a_t$ be a sequence of positive integers with $v+t\geq \frac{n+1}{2}$, $t\geq
1$ and $2\leq a_1\leq \cdots\leq a_t\leq v+\sum\limits_{i=1}^t
a_i\leq n-j,$ where $j$ is a positive integer. Then the following hold:
\begin{itemize}
  \item[(i)]  $v \geq a_t+\ldots + a_{t-j+1}-j+1;$
  \item [(ii)] For any integer $k\in [2,v+\sum\limits_{i=1}^t a_i],$
there exists a subsequence $T$ of $S$ with $|T|\ge 2$  and $\sigma (T)=k;$
  \item [(iii)]
 If $v+\sum\limits_{i=1}^t a_i\leq n-2,$ then for every
integer $k\in [a_1,v+\sum\limits_{i=2}^t a_i]$, there exist two
subsequences $T_1, T_2$ of $S$ with $\sigma(T_1)=\sigma(T_2)=k$ and
$|T_1|>|T_2|$.
\end{itemize}

\end{lemma}

\begin{proof}
We have clearly \begin{eqnarray*}n-j&\ge & v+\sum\limits_{i=1}^t
a_i\\&\geq &v+2(t-j)+\sum\limits_{i=t-j+1}^t
a_i\\&=&2(v+t)-2j-v+\sum\limits_{i=t-j+1}^t
a_i\ge n+1-2j-v+\sum\limits_{i=t-j+1}^t
a_i.
\end{eqnarray*}
Thus  (i) holds.

 By (i), we have $a_t\leq v$ and (ii) holds clearly for $k\le a_t$.  may Assume $k>a_t$.
Let $\ell$ be the maximal integer of $[1,t]$ such that
$\sum\limits_{i=1}^{\ell} a_i\leq k$ and put
$k^{'}=k-\sum\limits_{i=1}^{\ell} a_i$. Note that $k^{'}\leq
v$. Thus, $(\prod\limits_{i=1}^{\ell} a_i)\cdot 1^{k^{'}}$ is a
subsequence of $S$ of length at least two and of sum $=k$, proving (ii).

Let us prove (iii). Assume that $v+\sum\limits_{i=1}^t a_i\leq n-2.$
Since $a_t=\sigma(a_t)=\sigma (1^{a_t}),$ and since  $v\geq a_1$ by (i), we may assume that $t\geq 2$.

By (i), $v\geq a_{t-1}+a_t-1$.

Let $s$ be the maximal integer of $[1,t]$ such that
$\sum\limits_{i=1}^{s} a_i\leq k$.

For $s=t$, we have $k-\sum\limits_{i=2}^t a_i\leq (v+
\sum\limits_{i=2}^t a_i)-\sum\limits_{i=2}^t a_i =v$. Thus,
$(\prod\limits_{i=1}^t a_i)\cdot 1^{k-\sum\limits_{i=1}^t a_i}$ and
$(\prod\limits_{i=2}^t a_i)\cdot 1^{k-\sum\limits_{i=2}^t a_i}$ are
two subsequences of $S$ with sum $=k$ and of distinct lengths.

For $s<t$, we have  $$k-\sum\limits_{i=1}^{s-1} a_i\leq
(\sum\limits_{i=1}^{s+1} a_i-1)-\sum\limits_{i=1}^{s-1}
a_i=a_{s}+a_{s+1}-1\leq a_{t-1}+a_t-1\leq v.$$ Thus,
$(\prod\limits_{i=1}^{s-1} a_i)\cdot
1^{k-\sum\limits_{i=1}^{s-1} a_i}$ and
$(\prod\limits_{i=1}^{s} a_i)\cdot 1^{k-\sum\limits_{i=1}^{\ell}
a_i}$ are two subsequences of $S$ with sum $k$ and of distinct
lengths.\end{proof}

\begin{proofof}{Theorem}{\ref{Theorem:main resul}}

Suppose to the contrary of the theorem that $S$ assumes three
distinct values. Then   $0$ is not among the values of $S,$
otherwise $S\cdot 0^{-1}$ would be a modular zero-sum free
subsequence of $S$ with length $n-1$, and hence $S\cdot 0^{-1}$
assumes only one value,  by lemma \ref{n-1}, a contradiction.

We distinguish two cases.

{\bf Case 1}  $r\ge \frac{n}{2}.$

By lemma \ref{Lemma:Alon}, we have $r\le h(S).$

Put $S=a^{v} a_1\cdot \ldots \cdot a_t,$
where  $a_i\neq a,$ for $i=1, \ldots , t.$

By our assumption, we have  $t\geq 2.$

Assume first that  $\gcd(a,n)>1$. Thus  $h(S)\ge r\ge \frac{n}{2}\ge
\frac{n}{\gcd(a,n)}$. It forces that
$r=\frac{n}{\gcd(a,n)}=\frac{n}{2}$ and $\gcd(a,n)=2.$ It follows
that $2 \nmid a_i,$ for $i=1, \ldots , t.$ Otherwise $a_i=\ell a$
$\pmod n$ for some positive integer $2\le \ell \le
\frac{n}{\gcd(a,n)}=\frac{n}{2}$ and $a^{ \frac{n}{2}-\ell}\cdot
a_i$ would be a modular zero-sum subsequence with length $<r.$

 Since $\gcd(a,n)=2$ and all $a_i$ are odd, we have that $a_i+a_j=s_{ij}a $ $\pmod n$
  for any $i\ne j\in \{1,2, \ldots,  t\}$ for some $0\le s_{ij}\le \frac{n}{2}-1.$ Now
  $\frac{n}{2}=r=|(a_i\cdot a_j) 1^{\frac{n}2-s_{ij}}|=2+\frac{n}{2}-s_{ij}.$
It follows that $s_{ij}=2.$ Therefore $a_i+a_j\equiv 2a$ for any  $
i\in \{1,2, \ldots, t\}.$

If $t\ge 3$. Since $a_i+a_j\equiv 2a$ $\pmod n$ and thus (observing
that $a_i\in [1,n-1]$) $a_1=a_2= \ldots =a_{t},$  a contradiction.
So, we have $t=2$. But Now we have $a^{n-2}(a_1a_2)$ is also
zero-sum modulo $n$, a contradiction.

Therefore, we assume that $\gcd(a,n)=1.$ Thus for some $m$ coprime to $n,$ we have  $m*S=R=1^v b_1\cdot \ldots \cdot b_t,$ and $2\le b_1\le  \ldots  \le b_t \le n-1.$ Clearly every modular zero-sum subsequence of $R$ has length $=r.$

Now we shall show that
\begin{equation}\label{x<n-v}
b_t\leq n-v-1.
\end{equation}
Suppose to the contrary that $$b_t\geq n-v.$$ We must have $b_1\leq
n-v-1,$  since otherwise, $b_1\cdot 1^{n-b_1}$ and $b_t\cdot
1^{n-b_t}$ would  be two modular zero-sum subsequences of $R$ of
distinct lengths. Since $b_t\cdot 1^{n-b_t}$ is a modular zero-sum
subsequence of $S$, we have that $n-b_t+1=|b_t\cdot 1^{n-b_t}|\geq
\frac{n}{2}\geq n-v$, and so $b_t\leq v+1$.  Notice that
$b_1+b_t\leq v+1+(n-v-1)=n$. Thus, $b_1\cdot b_t\cdot 1^{n-b_1-b_t}$
and $b_t\cdot 1^{n-b_t}$ are two modular zero-sum subsequences of
$R$ of distinct lengths, a contradiction.

Choose a  subsequence $T$ of $R$ with $\sigma (T)\equiv 0 \ \mod n$ with a maximal number of values.
 Put $T=1^{\tau}\cdot x_1\cdot \ldots\cdot x_u$ and
$ST^{-1}=1^{\gamma}\cdot y_1\cdot\ldots\cdot y_w.$

We shall assume that $2\le x_1\le \ldots  \le x_u$ and that $2\le y_1\le \ldots  \le y_w.$

We must have
 $$x_1\ge \gamma +1,$$ otherwise $\sigma (1^{x_1+\tau}\cdot x_2\cdot \ldots\cdot x_u)\equiv 0\ \pmod n.$

Similarly $y_1\ge  \tau +1.$

 Clearly, $u\geq 1$. By \eqref{x<n-v} and since $|T|\le v,$ we have

\begin{eqnarray*} w
&=&|ST^{-1}|-\gamma \\&=&n-|T|-\gamma \geq n-v-\gamma
\\&\geq &  n-x_1-v+1 \\ &\ge & n-b_t-v+1\ge 2.
\end{eqnarray*}

By \eqref{x<n-v} and since $v\geq \frac{n}{2}$, we have that
$b_t<\frac{n}{2}.$
It follows that $$n>
x_u+y_w\ge x_1+y_1\ge \gamma +1+ \tau+1=v+2> n-v.$$

Since $x_1y_11^{n-x_1-y_1}$ and $x_uy_w1^{n-x_u-y_w}$ are modular
zero-sum subsequences, we conclude that $x_1=\cdots=x_u$ and
$y_1=\cdots=y_w$. Since $S$ has at least  $3$ distinct values, we
have $x_1\neq y_1$, thus, $T'=1^{n-x_1-y_1}\cdot x_1\cdot y_1$ is a
modular zero-sum subsequence of $R,$ with more distinct values than
$T,$ a contradiction.

{\bf Case 2} $r<\frac{n}{2}.$

Choose a modular  zero-sum subsequence $T$ of $S$. Then $ST^{-1}$ a
modular zero-sum free subsequence with $|ST^{-1}|>\frac{n}{2}$. By
Theorem \ref{Lemma:SavchevChen}, for some positive integer  $m$
coprime to $n,$ we have $m*(ST^{-1})=1^{\gamma}\cdot y_1\cdot
\ldots\cdot y_w,$ where $2\leq y_1\leq \cdots\leq
y_w<\gamma+\sum\limits_{i=1}^w y_i\leq n-1.$ Put $R=m*S.$ Clearly
every modular zero-sum subsequence of $R$ has length $=r.$ So
without loss of generality, we may take $m=1.$ Also, put
$T=1^{\tau}\cdot x_1\cdot \ldots\cdot x_u,$ where $2\leq x_1\leq
\cdots\leq x_u\leq n-1$.

We first note that
$$x_1\geq \gamma+1,$$
otherwise $1^{x_1}\cdot (x_1^{-1}T)$ is a zero-sum sequence of
length larger than $|T|$, a contradiction.

 We must have   $w\geq 1$. Otherwise,
$\gamma =|ST^{-1}|\ge \frac{n+1}{2}$ and hence $x_1\geq
\gamma+1>n-\gamma$. Therefore, $1^{n-x_i+\beta}\cdot (x_i^{-1}T)$ is
a modular zero-sum subsequence of $S$ for every $i=1, \cdots, u$.
This forces that $x_1=\cdots =x_u$, a contradiction on that $S$
takes at least three distinct values.

We must have $u\geq 2,$ since otherwise (observing that $u\neq 0$),

$x_1=n-\tau=n-|T|+1=|ST^{-1}|+1\leq \gamma+\sum\limits_{i=1}^w y_i$.
By Lemma \ref{Lemma:Our lemma} (ii) with $j=1$, there is a
subsequence $U$ of $ST^{-1}$ with $|U|\ge 2$ such that $x_1=\sigma
(U).$ Now $1^{\tau}x_1$ and
 $1^{\tau}U$ are modular zero-sum subsequences with distinct lengths,
 a
contradiction.

Thus,
\begin{equation} \label{equw}
 w\geq 1 \ \mbox{ and} \  u\geq 2.\end{equation}

Let $X_{\ell}$  be the unique integer of $[0,n-1]$ such that
$$X_{\ell}\equiv\sum\limits_{i=1}^{\ell}x_i\pmod n$$ for
$\ell=1,\ldots,u$.

Applying Lemma \ref{Lemma:Our lemma} (ii), we have that
\begin{equation}\label{bound for x_1}
x_1\geq \gamma+\sum\limits_{i=1}^w y_i+1,
\end{equation}
and so $$\gamma+\sum\limits_{i=1}^w y_i\leq x_1-1 \leq n-2.$$

By Lemma \ref{Lemma:Our lemma} (iii), we have that
\begin{equation}\label{sums of T}
\sum(T)\cap [y_1,\gamma+\sum\limits_{i=2}^w y_i]=\emptyset,
\end{equation}
where $\sum(T)$ denotes the set of the sums of the nonempty subsequences of $T.$

By \eqref{bound for x_1}, we have that
\begin{equation}\label{x_i>n/2+2}
x_i\geq x_1\geq \gamma+\sum\limits_{i=1}^w y_i+1\geq
|ST^{-1}|+1>\frac{n}{2}+1
\end{equation}

for $i=1,\ldots,u$, which implies
\begin{equation}\label{sums of two elements}
x_{i_1}+x_{i_2}\not\equiv 1,2\pmod n
\end{equation}
for any $1\leq i_1<i_2\leq u$. By lemma \ref{Lemma:Our lemma} (i),
we see that

\begin{equation}\label{the number of 1}
\gamma\geq y_1.
\end{equation}

Therefore, by \eqref{sums of T}, \eqref{sums of two elements} and
\eqref{the number of 1} we conclude that $X_2\notin
[1,\gamma+\sum\limits_{i=2}^w y_i]$, i.e., $x_1+x_2-n=X_2\geq
\gamma+\sum\limits_{i=2}^w y_i+1\geq
\gamma+w=|ST^{-1}|\geq\frac{n+1}{2}$. It follows that $x_2\geq
\frac{x_1+x_2}{2}\geq\frac{3n+1}{4}=n-\frac{n-1}{4}$, i.e.,
\begin{equation}\label{x_2>n-n/4}
x_2\geq \lceil n-\frac{n-1}{4}\rceil=n-\lfloor \frac{n-1}{4}\rfloor.
\end{equation}

Now we shall show that
\begin{equation}\label{the upper bound for x_u}
x_u\leq n-\tau-3.
\end{equation}
Since $u\geq 2$, we have $x_u\leq n-\tau-1$. Suppose $x_u\in
\{n-\tau-1,n-\tau-2\}$. Then $X_{u-1}\in \{1,2\}$. By
\eqref{x_i>n/2+2} and \eqref{sums of two elements}, we have $u-1\geq
3$. By \eqref{the number of 1}, $\gamma\geq y_1\geq 2$, thus,
$T\cdot(\prod\limits_{i=1}^{u-1} x_i)^{-1} \cdot 1^{X_{u-1}}$ is a
zero-sum subsequence of $S$ with length $|T|-(u-1)+X_{u-1}\leq
|T|-3+2=|T|-1$, a contradiction. Therefore, $x_u\leq n-\tau-3$.

Let $t\in [1,u]$ be the largest integer such that
$X_i>\lceil\frac{n-1}{4}\rceil$ for every $i\in \{1,\ldots,t\}$. By
\eqref{x_2>n-n/4} and \eqref{the upper bound for x_u}, we see that
$n-\lfloor\frac{n-1}{4}\rfloor\leq x_i\leq n-3$ for $i=2,\ldots,u$.
It follows that
\begin{equation}\lceil\frac{n-1}{4}\rceil<
X_{\ell}=X_{\ell-1}+x_{\ell}-n\leq X_{\ell-1}-3
\end{equation}
for $\ell=2,3,\ldots,t$. Put
$$q=\min(\lceil\frac{u+1}{3}\rceil,t).$$

We shall show that $$X_q\leq  \gamma+\sum\limits_{i=2}^w y_i.$$ If
$q=\lceil\frac{u+1}{3}\rceil\leq t$, then by (10) $X_q\leq
X_1-3(q-1)=x_1-3(q-1)\leq n-\tau-3-3\lceil\frac{u+1}{3}\rceil+3\leq
n-\tau-u-1=|ST^{-1}|-1\leq \gamma+\sum\limits_{i=2}^w y_i$. If
$q=t<\lceil\frac{u+1}{3}\rceil\leq u$, then $X_{t+1}\leq
\lceil\frac{n-1}{4}\rceil$, which implies that
$X_t=X_{t+1}+(n-x_{t+1})\leq X_{t+1}+\lfloor\frac{n-1}{4}\rfloor\leq
\lceil\frac{n-1}{2}\rceil\leq |ST^{-1}|-1\le
\gamma+\sum\limits_{i=2}^w y_i$.

Thus, by \eqref{sums of T} and \eqref{the number of 1}, we have that
$X_q<y_1\leq \gamma$, thus, $T\cdot (\prod\limits_{i=1}^q
x_i)^{-1}\cdot 1^{X_q}$ is a zero-sum subsequence of $S$ with length
$|T|-q+X_q>
|T|-\lceil\frac{u+1}{3}\rceil+\lceil\frac{n-1}{4}\rceil\geq |T|-
\lceil\frac{|T|+1}{3}\rceil+\lceil\frac{n-1}{4}\rceil\geq |T|-
\lceil\frac{n+1}{6}\rceil+\lceil\frac{n-1}{4}\rceil\geq |T|$, a
contradiction.
\end{proofof}

\bigskip

\noindent {\bf Acknowledgement.} This research has been supported in
part by the 973 Project, the PCSIRT Project of the Ministry of
Education, the Ministry of Science and Technology, the National
Science Foundation of China. This paper was completed partly during
a visit by the first author to University of  P. et M. Curie in
France. He would like to thank the host institution for its kind
hospitality.


\begin{thebibliography}{99}

\bibitem{Alon} N. Alon, {\it Subset sums},  J. Number Theory,
27(1987) 196-205.

\bibitem{EG} P. Erd\H{o}s and R.L. Graham, {\it Old and new problems and results in combinatorial number
theory}.
 Monographies de L'Enseignement Mathématique, 28. Université de Genève, L'Enseignement Mathématique, Geneva, 1980. 128 pp. (Reviewer: L. C. Eggan)

\bibitem{EH} P. Erd\H{o}s and H. Heilbronn,
{\it On the Addition of residue classes mod  $p$},  Acta Arith.
9(1964), 149-159.

\bibitem{ES} P. Erd\H{o}s and E. Szemer\'edi,
{\it On a problem of Graham},  Publ. Math. Debrecen, 23(1976), no.
1-2, 123-127.

\bibitem{gerlodinger} A. Geroldinger, F. Halter-Koch, {\it Non-unique factorizations.
Algebraic, combinatorial and analytic theory. Pure and Applied
Mathematics (Boca Raton)}, 278. Chapman \& Hall/CRC, Boca Raton, FL,
2006. xxii+700 pp.

\bibitem{SC} S. Savchev and F. Chen, {\it Long zero-free sequences in finite cyclic groups},  Discrete
Mathematics, 307(2007) 2671-2679.
\bibitem{sch}{ P. Scherk}, { Distinct elements in a set of sums},  Amer. Math. Monthly, 62 (1955), pp. 46--47.
\bibitem{Yuan} P.Z. Yuan, {\it On the index of minimal zero-sum sequences
over finite cyclic groups},  J. Combin. Theory Ser. A, 114(2007)
1545-1551.


\end{thebibliography}
\end{document}